\documentclass{amsart}

\usepackage{amsmath}
\usepackage{amssymb}
\usepackage{amsthm}
\usepackage{graphicx}
\usepackage{enumitem}
\usepackage{refcount}

\theoremstyle{plain}
\newtheorem*{thm*}{Theorem}
\newtheorem{thm}{Theorem}

\newtheorem{lem}[thm]{Lemma}

\theoremstyle{definition}
\newtheorem{defn}[thm]{Definition}

\newcommand{\id}{\mathrm{id}}

\newcommand{\e}{\varepsilon}
\newcommand{\0}{\emptyset}
\newcommand{\sm}{\setminus}

\newcommand{\ol}{\overline}

\begin{document}

\title[Hereditarily equivalent plane continua]{A complete classification of\\hereditarily equivalent plane continua}
\author{L. C. Hoehn \and L. G. Oversteegen}
\date{\today}

\address[L.\ C.\ Hoehn]{Nipissing University, Department of Computer Science \& Mathematics, 100 College Drive, Box 5002, North Bay, Ontario, Canada, P1B 8L7}
\email{loganh@nipissingu.ca}

\address[L.\ G.\ Oversteegen]{University of Alabama at Birmingham, Department of Mathematics, Birmingham, AL 35294, USA}
\email{overstee@uab.edu}

\thanks{The first named author was partially supported by NSERC grant RGPIN 435518}
\thanks{The second named author was partially supported by NSF-DMS-1807558}


\begin{abstract}
A continuum is hereditarily equivalent if it is homeomorphic to each of its non-degenerate sub-continua.  We show in this paper that the arc and the pseudo-arc are the only non-degenerate hereditarily equivalent plane continua.
\end{abstract}

\maketitle

\section{Introduction}
\label{sec:intro}

By a \emph{continuum}, we mean a compact connected metric space.  A continuum is \emph{non-degenerate} if it contains more than one point.  We refer to the space $\mathbb{R}^2$, with the Euclidean topology, as \emph{the plane}.  The Euclidean distance between two points $x,y$ in $\mathbb{R}^2$ (or $\mathbb{R}^3$) will be denoted $\|x - y\|$.  An \emph{arc} is a space which is homeomorphic to the interval $[0,1]$.  By a \emph{map} we mean a continuous function.

A continuum $X$ is \emph{hereditarily equivalent} if it is homeomorphic to each of its non-degenerate subcontinua.  This concept was introduced by Mazurkiewicz, who was interested in topological characterizations of the arc.  In the second volume of Fundamenta Mathematicae in 1921, Mazurkiewicz \cite{mazurkiewicz21} asked (Probl\`{e}me 14) whether the arc is the only non-degenerate hereditarily equivalent continuum.

A continuum $X$ is \emph{decomposable} if it is the union of two proper subcontinua, and \emph{indecomposable} otherwise.  $X$ is \emph{hereditarily indecomposable} if every subcontinuum of $X$ is indecomposable.  $X$ is \emph{arc-like} (respectively, \emph{tree-like}) if for every $\e > 0$ there exists an $\e$-map from $X$ to $[0,1]$ (respectively, to a tree), where $f: X \to Y$ is an \emph{$\e$-map} if for each $y \in Y$ the preimage $f^{-1}(y)$ has diameter less than $\e$.  Henderson \cite{henderson60} showed that the arc is the only decomposable hereditarily equivalent continuum.  Cook \cite{cook68} has shown that every hereditarily equivalent continuum is tree-like.

Probl\`{e}me 14 of Mazurkiewicz was formally answered by Moise \cite{moise48} in 1948, who constructed another hereditarily equivalent plane continuum which he called the ``pseudo-arc'', due to this property it has in common with the arc.  The pseudo-arc is a one-dimensional fractal-like hereditarily indecomposable arc-like continuum.  Such a space was constructed by Knaster \cite{knaster22} in 1922, and another by Bing \cite{bing48} in 1948 which he proved was topologically homogeneous.  Bing \cite{bing51} proved in 1951 that the pseudo-arc is the only hereditarily indecomposable arc-like continuum.  From this characterization it follows that the spaces of Knaster, Moise, and Bing are all homeomorphic, and also it can immediately be seen that the pseudo-arc is hereditarily equivalent.

Since Moise's article, the question has been: What are all hereditarily equivalent continua?  The main result of this paper is:
\begin{thm}
\label{thm:main1}
If $X$ is a non-degenerate hereditarily equivalent plane continuum, then $X$ is homeomorphic to the arc or to the pseudo-arc.
\end{thm}
It remains an open question whether there exists any other hereditarily equivalent continuum in $\mathbb{R}^3$.

As part of the sequel we will also give a new characterization (Theorem \ref{thm:charsep}) of the pseudo-arc.

\section{Plane strips}
\label{sec:strips}

If a continuum admits an $\e$-map to an arc then it can be covered by a chain of open sets whose diameters are less than $\e$ (i.e.\ a set that roughly looks like a tube of small diameter).  The notion of an $\e$-strip (see Definition \ref{defn:strip} below), introduced in \cite{OT82} in a slightly different form, conveys a similar feeling.  However, it was observed in \cite[Figure 1]{OT82} that, for arbitrarily small $\e > 0$, there exists an $\e$-strip which does not admit a $1$-map to an arc.  Nevertheless we show in this paper (Theorem \ref{thm:hered indec strip} below) that if a hereditarily indecomposable plane continuum is contained in an $\e$-strip for arbitrarily small $\e > 0$, then it must in fact be homeomorphic to the pseudo-arc.

Given two points $x,y$ in the plane $\mathbb{R}^2$ we denote by $\ol{xy}$ the straight line segment joining them.  Given points $v_1,\ldots,v_n$ in $\mathbb{R}^2$, the \emph{polygonal arc} $A$ with vertices $v_1,\ldots,v_n$ is the union of the straight line segments $\ol{v_1v_2}, \ldots, \ol{v_{n-1}v_n}$.  Denote the \emph{vertex set} of $A$ by $V_A = \{v_1,\ldots,v_n\}$.  If $v_n = v_1$, then we call $A$ a \emph{polygonal closed curve}.  We will need the following lemma which was proved in \cite{OT82}.

\begin{lem}[\cite{OT82}, Lemma 2.1]
Let $T$ be a polygonal closed curve in $\mathbb{R}^2$ with vertex set $V_T$.  Given any $z \in \mathbb{R}^2 \sm T$ we say $z$ is \emph{odd} (respectively, \emph{even}) with respect to $T$ if there exists a polygonal arc $A$ with vertex set $V_A$ from $z$ to a point in the unbounded component of $\mathbb{R}^2 \sm T$ so that $A \cap V_T = \0 = T \cap V_A$ and $|T \cap A|$ is odd (respectively, even).  Then this notion of odd/even is well-defined, i.e.\ independent of the choice of $A$.
\end{lem}

A component $U$ of $\mathbb R^2 \sm T$ is called \emph{odd} (respectively, \emph{even}) if each point of $U$ is odd (respectively, \emph{even}) with respect to $T$.  Clearly the unbounded complementary domain of $T$ is even.

A map $f: [0,1] \to \mathbb{R}^2$ is \emph{piecewise linear} if there are finitely many points $0 = t_1 < t_2 < \ldots < t_n = 1$ such that for each $i = 1,\ldots,n-1$, as $t$ runs from $t_i$ to $t_{i+1}$, $f(t)$ parameterizes the straight line segment $\overline{f(t_i)f(t_{i+1})}$.  If $f$ is a piecewise linear map, then clearly $f([0,1])$ is a polygonal arc.

\begin{defn}[$\e$-strip]
\label{defn:strip}
Suppose that $f,g: [0,1] \to \mathbb R^2$ are two piecewise linear maps into the plane such that $f([0,1]) \cap g([0,1]) = \0$ and for all $t \in [0,1]$, $\|f(t) - g(t)\| < \e$.  Let $B_t = \ol{f(t)g(t)}$, and let $T_t = B_0 \cup f([0,t]) \cup g([0,t] \cup B_t$.  We denote the union of all odd (respectively, even) complementary domains of $T_t$ by $S^-_t$ (respectively, $S^+_t$).  If $B_0 \cap B_1 = \0$, then we say that $S_1^-$ is an \emph{$\e$-strip with disjoint ends}.
\end{defn}

See Figure \ref{fig:strip} for an illustration of a simple $\e$-strip.

\begin{figure}
\begin{center}
\includegraphics[scale=0.8]{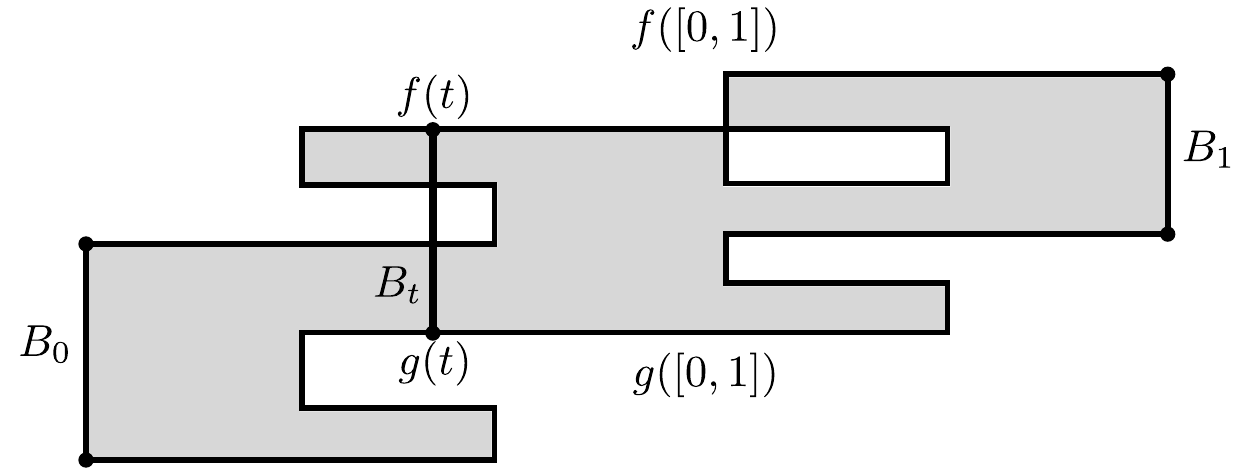}
\end{center}

\caption{An illustration of an $\e$-strip, with a generic bridge $B_t$ drawn.  This strip has one odd domain, which is shaded gray.}
\label{fig:strip}
\end{figure}

We say a continuum $X$ is contained in an $\e$-strip with disjoint ends if there exist such $f,g$ as in the above definition such that $X \subset S^-_1$.  Observe that in this situation, $X \cap B_0 = \0 = X \cap B_1$.

If $X$ is an indecomposable and hereditarily equivalent plane continuum, then it contains uncountably many pairwise disjoint copies of itself.  In particular it contains a copy of $X \times C$, where $C$ is the Cantor set \cite{vandouwen93}.  This is the key observation behind the following result.

\begin{lem}[\cite{OT84a}, Theorem 15]
\label{lem:strip}
Suppose that $X$ is a non-degenerate, indecomposable and hereditarily equivalent plane continuum.  Then there exists a non-degenerate subcontinuum $Y$ such that for each $\e > 0$, $Y$ is contained in an $\e$-strip with disjoint ends.
\end{lem}

\section{Separators}
\label{sec:separators}

In light of Lemma \ref{lem:strip} above, to prove Theorem \ref{thm:main1} it suffices to show that any hereditarily indecomposable continuum $X$ contained in arbitrarily small plane strips is homeomorphic to the pseudo-arc (Theorem \ref{thm:hered indec strip}).  Our strategy below is to consider a small strip containing the continuum $X$, and to approximate $X$ by a graph $G$ contained in that strip.  If we vary $t$ from $0$ to $1$, the bridge $B_t$ in the strip sweeps across the graph $G$.  As it does so, it may wander back and forth in $G$, in a pattern whose essential property is captured in the following result.  We will then use the crookedness of the hereditarily indecomposable continuum $X$ to match with that pattern (see Theorems \ref{thm:lift} and \ref{thm:charsep} below) to obtain an $\e$-map to an arc.

\begin{lem}
\label{lem:separate}
Suppose that a graph $G$ is contained in an $\e$-strip $S_1^-$ with disjoint ends.  Let
\[C = \{(x,t) \in G \times [0,1]: x \in B_t\} .\]
Then $C$ separates $G \times \{0\}$ from $G \times \{1\}$ in $G \times [0,1]$.
\end{lem}

\begin{proof}
Define the function $\varphi: G \times [0,1] \to \mathbb{R}$ by
\[ \varphi(x,t) = \begin{cases}
+d(x,B_t) & \textrm{if } x \in S^+_t \\
-d(x,B_t) & \textrm{if } x \in S^-_t \\
0 & \textrm{otherwise,}
\end{cases} \]
where $d(x,B_t) = \inf\{\|x - b\|: b \in B_t\}$.  Then $\varphi$ is  a continuous function (see the proof of Lemma 2.3 in \cite{OT82}).  Since $S^-_0 = \0 = B_0\cap G$, $\varphi(x,0) > 0$ for each $x \in G$.  Similarly, since $G \subset S^-_1$, $\varphi(x,1) < 0$ for each $x \in G$.  Hence the set of points $C$ where $\varphi(x,t) = 0$ must separate $G \times \{0\}$ from $G \times \{1\}$ in $G \times [0,1]$.
\end{proof}

In \cite[Theorem 20]{HO14}, the authors gave a characterization of hereditarily indecomposable continua in terms of sets as in Lemma \ref{lem:separate} which separate $G \times \{0\}$ from $G \times \{1\}$ in the product $G \times [0,1]$ of a graph $G$ with $[0,1]$.  Here we give a simplified version of that theorem, which is more broadly applicable.

\begin{thm}
\label{thm:lift}
A continuum $X$ is hereditarily indecomposable if and only if for any map $f: X \to G$ to a graph $G$, and for any open set $U \subseteq G \times (0,1)$ which separates $G \times \{0\}$ from $G \times \{1\}$ in $G \times [0,1]$, there exists a map $h: X \to U$ such that $f = \pi_1 \circ h$ (where $\pi_1: G \times [0,1] \to G$ is the first coordinate projection).
\end{thm}

\begin{proof}
According to \cite[Theorem 20]{HO14}, a continuum $X$ is hereditarily indecomposable if and only if for any map $f: X \to G$ to a graph $G$ with metric $d$, for any set $M \subseteq G \times (0,1)$ which separates $G \times \{0\}$ from $G \times \{1\}$ in $G \times [0,1]$, for any open set $U \subseteq G \times [0,1]$ with $M \subseteq U$, and for any $\e > 0$, there exists a map $h: X \to U$ such that $d(f(x), \pi_1 \circ h(x)) < \e$ for all $x \in X$.  The condition in the present theorem is clearly stronger than this condition from \cite[Theorem 20]{HO14}.  Therefore, to prove the present theorem we need only consider the forward implication.

Suppose $X$ is hereditarily indecomposable, let $f: X \to G$ be a map to a graph $G$ with metric $d$, and let $U \subset G \times (0,1)$ be an open set which separates $G \times \{0\}$ from $G \times \{1\}$ in $G \times [0,1]$.  It is well-known (see e.g.\ \cite[Theorem \S 46.VII.3]{kuratowski68}) that there exists a closed set $M \subset U$ which also separates $G \times \{0\}$ from $G \times \{1\}$ in $G \times [0,1]$.  Let $\e > 0$ be small enough so that the open set
\begin{align*}
U_1 = \{(g,t) \in G \times [0,1]: & \textrm{ there exists } (g',t') \in M \textrm{ such that } \\
& d(g,g') < \e \textrm{ and } |t-t'| < \e\}
\end{align*}
is contained in $U$.  Let
\begin{align*}
U_2 = \{ (g,t) \in G \times [0,1]: & \textrm{ there exists } (g',t') \in M \textrm{ such that } \\
& d(g,g') < \tfrac{\e}{2} \textrm{ and } |t-t'| < \e\} ,
\end{align*}
and apply \cite[Theorem 20]{HO14} to obtain a map $h': X \to U_2$ such that $d(f(x), \pi_1 \circ h'(x)) < \frac{\e}{2}$ for all $x \in X$.

Define $h: X \to U$ by $h(x) = (f(x), \pi_2 \circ h'(x))$, where $\pi_2: G \times [0,1] \to [0,1]$ is the second coordinate projection.  Clearly this function $h$ is continuous, and $f = \pi_1 \circ h$.  To see that the range of $h$ is really contained in $U$, let $x \in X$, and denote $h'(x) = (g,t)$, so that $h(x) = (f(x),t)$.  Because $h'(x) \in U_2$, there exists $(g',t') \in M$ such that $d(g,g') < \frac{\e}{2}$ and $|t-t'| < \e$.  Moreover, by choice of $h'$ we have $d(f(x),g) < \frac{\e}{2}$.  So by the triangle inequality, we have $d(f(x),g') < \e$, which means $h(x) \in U_1 \subseteq U$, as desired.
\end{proof}

By Bing's \cite{bing51} result a hereditarily indecomposable continuum is homeomorphic to the pseudo-arc if and only if it is arc-like.  A new characterization of the pseudo-arc, involving the notion of \emph{span zero} (see \cite{lelek64}), was obtained in \cite{HO14}.  It states that a hereditarily indecomposable continuum is a pseudo-arc if and only if it has span zero.  The more technical characterization of the pseudo-arc in Theorem \ref{thm:charsep} below is useful in cases when (like in the case of hereditarily equivalent plane continua) it is not a priori known that $X$ has span zero.

In the statement below we assume that all spaces (i.e., $X$, $G$, and $I$) are contained in Euclidean space $\mathbb{R}^3$.  One could just as well use the Hilbert cube $[0,1]^{\mathbb{N}}$, depending on the intended application.

\begin{thm}
\label{thm:charsep}
Suppose that $X \subset \mathbb{R}^3$ is a hereditarily indecomposable continuum.  Then the following are equivalent:
\begin{enumerate}
\item $X$ is homeomorphic to the pseudo-arc;
\item For each $\e > 0$ there exist a graph $G \subset \mathbb{R}^3$, a map $f: X \to G$ with $\|(x - f(x)\| < \e$ for each $x \in X$, and an arc $I \subset \mathbb{R}^3$ with endpoints $a$ and $b$, such that the set
\[ U = \{(x,t) \in G \times (I \sm \{a,b\}): \|x - t\| < \e\} \]
separates $G \times \{a\}$ from $G \times\{b\}$ in $G \times I$.
\end{enumerate}
\end{thm}

\begin{proof}
Suppose $X$ is homeomorphic to the pseudo-arc, and fix $\e > 0$.  Note $X$ is arc-like and, hence \cite{lelek64}, $X$ has span zero.  Therefore, according to Theorem 4 of \cite{HO14}, there exists $\delta > 0$ such that for any graph $G \subset \mathbb{R}^3$ and arc $I \subset \mathbb{R}^3$ both within Hausdroff distance $\delta$ from $X$, the set $U = \{(x,t) \in G \times I: \|x - y\| < \e\}$ separates $G \times \{a\}$ from $G \times \{b\}$ in $G \times I$, where $a,b$ are the endpoints of $I$.  We may assume that $\delta < \e$.  Since $X$ is arc-like, we may choose an arc $G \subset \mathbb{R}^3$ within Hausdorff distance $\delta$ of $X$ and a map $f: X \to G$ such that $\|x - f(x)\| < \e$ for all $x \in X$.  Choose any arc $I$ within Hausdorff distance $\delta$ from $X$.  Then $G$, $f$, and $I$ satisfy the conditions of statement (2), as desired.

Conversely, suppose statement (2) holds.  To prove that $X$ is homeomorphic to the pseudo-arc, by \cite{bing51} it suffices to show that for each $\e > 0$ there exists an $\e$-map from $X$ to an arc.  Fix $\e > 0$.  Suppose that $G \subset \mathbb{R}^3$ is a graph, $f: X \to G$ is a map such that $\|x - f(x)\| < \frac{\e}{4}$ for all $x \in X$, $I \subset \mathbb{R}^3$ is an arc with endpoints $a$ and $b$, and
\[ U = \left\{ (x,t) \in G \times (I \sm \{a,b\}): \|x - t\| < \frac{\e}{4} \right\} \]
separates $G \times \{a\}$ from $G \times \{b\}$ in $G \times I$.  Denote by $\pi_1: G \times I \to G$ the first coordinate projection and by $\pi_2: G \times I \to I$ the second coordinate projection.  By Theorem \ref{thm:lift} there exists a map $h: X \to U$ such that $f = \pi_1 \circ h$.  We claim that $\pi_2 \circ h(x): X \to I$ is an $\e$-map.
To see this suppose that $\pi_2 \circ h(x_1) = \pi_2 \circ h(x_2)$.  Then
\begin{align*}
\|x_1 - x_2\| &\leq \|x_1 - f(x_1)\| + \|\pi_1 \circ h(x_1) - \pi_2 \circ h(x_1)\| + {} \\
& \qquad + \|\pi_2 \circ h(x_2) - \pi_1 \circ h(x_2)\| + \|f(x_2) - x_2\| < \e .
\end{align*}
\end{proof}

\section{Proof of main result}
\label{sec:pseudo-arc}

We now apply the results established above to prove the following key theorem.

\begin{thm}
\label{thm:hered indec strip}
Let $X \subset \mathbb{R}^2$ be a hereditarily indecomposable plane continuum such that for each $\e > 0$, there is an $\e$-strip with disjoint ends containing $X$.  Then $X$ is homeomorphic to the pseudo-arc.
\end{thm}

\begin{proof}
Let $\e > 0$, and consider an $\frac{\e}{2}$-strip with disjoint ends containing $X$.  That is, consider piecewise linear maps $f,g: [0,1] \to \mathbb{R}^2$ such that $f([0,1]) \cap g([0,1]) = \0$, $\|f(t) - g(t)\| < \frac{\e}{2}$ for each $t \in [0,1]$, and $X \subset S^-_1$.  Identify $\mathbb{R}^2$ with $\mathbb{R}^2 \times \{0\} \subset \mathbb{R}^3$, and adjust $f$ slightly to obtain a map $f': [0,1] \to \mathbb{R}^3$ which is one-to-one (so that $f'([0,1])$ is an arc) and $\|f(t) - f'(t)\| < \frac{\e}{2}$ for all $t \in [0,1]$.

Clearly $X$ is $1$-dimensional, so there exists a graph $G \subset S_1^-$ and a map $h: X \to G$ such that $\|x - h(x)\| < \e$ for all $x \in X$.  By Lemma \ref{lem:separate}, the set
\[ C = \{(x,t) \in G \times [0,1]: x \in B_t\} \]
separates $G \times \{0\}$ from $G \times \{1\}$ in $G \times [0,1]$.  Clearly this set $C$ is contained in
\[ U = \left\{ (x,t) \in G \times (0,1): \|x - f(t)\| < \frac{\e}{2} \right\} ,\]
and the image of this set $U$ under the homeomorphism $\id \times f': G \times [0,1] \to G \times f'([0,1])$ is contained in
\[ U' = \{(x,y) \in G \times f'((0,1)): \|x - y\| < \e\} .\]
Therefore $U'$ separates separates $G \times \{f'(0)\}$ from $G \times \{f'(1)\}$ in $G \times f'([0,1])$.  Hence, by Theorem \ref{thm:charsep}, $X$ is homeomorphic to the pseudo-arc.
\end{proof}

We are now ready to prove our main result, Theorem \ref{thm:main1}.

\begin{proof}[Proof of Theorem \ref{thm:main1}]
Let $X$ be a non-degenerate hereditarily equivalent plane continuum.  If $X$ is decomposable, then $X$ is an arc by \cite{henderson60}.  Suppose then that $X$ is indecomposable, and hence hereditarily indecomposable.  By Lemma \ref{lem:strip}, we may assume that $X$ is embedded in the plane so that for each $\e > 0$, $X$ is contained in an $\e$-strip with disjoint ends.  It then follows from Theorem \ref{thm:hered indec strip} that $X$ is homeomorphic to the pseudo-arc.
\end{proof}

\bibliographystyle{amsalpha}
\bibliography{her-equivalent}

\providecommand{\bysame}{\leavevmode\hbox to3em{\hrulefill}\thinspace}
\providecommand{\MR}{\relax\ifhmode\unskip\space\fi MR }
\providecommand{\MRhref}[2]{%
  \href{http://www.ams.org/mathscinet-getitem?mr=#1}{#2}
}
\providecommand{\href}[2]{#2}
\begin{thebibliography}{Maz21}

\bibitem[Bin48]{bing48}
R.~H. Bing, \emph{A homogeneous indecomposable plane continuum}, Duke Math. J.
  \textbf{15} (1948), 729--742. \MR{0027144 (10,261a)}

\bibitem[Bin51]{bing51}
\bysame, \emph{Concerning hereditarily indecomposable continua}, Pacific J.
  Math. \textbf{1} (1951), 43--51. \MR{0043451 (13,265b)}

\bibitem[Coo70]{cook68}
H.~Cook, \emph{Tree-likeness of hereditarily equivalent continua}, Fund. Math.
  \textbf{68} (1970), 203--205. \MR{0266164 (42 \#1072)}

\bibitem[Hen60]{henderson60}
George~W. Henderson, \emph{Proof that every compact decomposable continuum
  which is topologically equivalent to each of its nondegenerate subcontinua is
  an arc}, Ann. of Math. (2) \textbf{72} (1960), 421--428. \MR{0119183 (22
  \#9949)}

\bibitem[HO16]{HO14}
Logan~C. Hoehn and Lex~G. Oversteegen, \emph{A complete classification of
  homogeneous plane continua}, Acta Math. \textbf{216} (2016), no.~2, 177--216.
  \MR{3573330}

\bibitem[Kna22]{knaster22}
B.~Knaster, \emph{Un continu dont tout sous-continu est ind\'{e}composable},
  Fund. Math. \textbf{3} (1922), 247--286.

\bibitem[Kur68]{kuratowski68}
K.~Kuratowski, \emph{Topology. {V}ol. {II}}, New edition, revised and
  augmented. Translated from the French by A. Kirkor, Academic Press, New York,
  1968.

\bibitem[Lel64]{lelek64}
A.~Lelek, \emph{Disjoint mappings and the span of spaces}, Fund. Math.
  \textbf{55} (1964), 199--214. \MR{0179766 (31 \#4009)}

\bibitem[Maz21]{mazurkiewicz21}
S.~Mazurkiewicz, \emph{Probl\`{e}me 14}, Fund. Math. \textbf{2} (1921), 286.

\bibitem[Moi48]{moise48}
Edwin~E. Moise, \emph{An indecomposable plane continuum which is homeomorphic
  to each of its nondegenerate subcontinua}, Trans. Amer. Math. Soc.
  \textbf{63} (1948), 581--594. \MR{0025733 (10,56i)}

\bibitem[OT82]{OT82}
Lex~G. Oversteegen and E.~D. Tymchatyn, \emph{Plane strips and the span of
  continua. {I}}, Houston J. Math. \textbf{8} (1982), no.~1, 129--142.
  \MR{666153 (84h:54030)}

\bibitem[OT84]{OT84a}
\bysame, \emph{Plane strips and the span of continua. {II}}, Houston J. Math.
  \textbf{10} (1984), no.~2, 255--266. \MR{744910 (86a:54042)}

\bibitem[vD93]{vandouwen93}
Eric~K. van Douwen, \emph{Uncountably many pairwise disjoint copies of one
  metrizable compactum in another}, Topology Appl. \textbf{51} (1993), no.~2,
  87--91. \MR{1229705}

\end{thebibliography}

\end{document}